\documentclass[a4paper, twoside]{article} 

\usepackage[
	top = 1.15 in, 
	bottom = 1.25 in,
	left = 1.15 in, 
	right = 1.15 in,
	includehead]{geometry}
\usepackage{scrextend}
\changefontsizes{11 pt}

\usepackage[all]{nowidow}
\usepackage{mathptmx}
\usepackage[T1]{fontenc}

\usepackage[utf8x]{inputenc}
\usepackage{amsfonts,amssymb,amsmath}
\usepackage{amsrefs}
\usepackage{url}
\usepackage{lipsum}
\BibSpec{article}{%
    +{}  {\PrintAuthors}                {author}
    +{,} { \textit}                     {title}
    +{.} { }                            {part}
    +{:} { \textit}                     {subtitle}
    +{,} { \PrintContributions}         {contribution}
    +{.} { \PrintPartials}              {partial}
    +{,} { }                            {journal}
    +{}  { \textbf}                     {volume}
    +{}  { \PrintDatePV}                {date}
    +{,} { \issuetext}                  {number}
    +{,} { \eprintpages}                {pages}
    +{,} { }                            {status}
    +{,} { \url}                        {url}    
    +{,} { \PrintDOI}                   {doi}
    +{,} { available at \eprint}        {eprint}
    +{}  { \parenthesize}               {language}
    +{}  { \PrintTranslation}           {translation}
    +{;} { \PrintReprint}               {reprint}
    +{.} { }                            {note}
    +{.} {}                             {transition}
    +{}  {\SentenceSpace \PrintReviews} {review}
}

\usepackage[bb=ams, cal=euler, scr=rsfso , frak=euler]{mathalpha}

\usepackage{graphicx}
\usepackage[
	british]{babel}
\usepackage{caption}
\usepackage{mathdots}
\usepackage{mathtools} 
\usepackage{mathbbol}
\usepackage{amsmath}
\usepackage{amsfonts}
\usepackage{amssymb}
\usepackage{booktabs}

\renewenvironment{abstract}{%
\small\begin{center}
\begin{minipage}{.9\textwidth}
}
{\par\noindent\end{minipage}\end{center}\vspace{3 em}}
\makeatletter
\renewcommand\@maketitle{%
\hfill
\begin{center}\begin{minipage}{0.9 	\textwidth}
\centering
\vskip 2em
\let\footnote\thanks 
{\LARGE \@title \par }
\vspace{1 em}
\vskip 1 em
{\large \@author \par}
\vspace{3.5 em}

\end{minipage}\end{center}
\par
}
\makeatother
 \usepackage{fancyhdr}
\newcommand{\stirlingi}{\genfrac{[}{]}{0pt}{}}

\newenvironment{tenumerate}{
\begin{enumerate}
  \setlength{\itemsep}{0pt}
  \setlength{\parskip}{0pt}
}{\end{enumerate}}

\usepackage[utf8x]{inputenc}
\usepackage{amsfonts,amssymb,amsmath}
\usepackage{graphicx}
\usepackage{mathtools}
\usepackage{mathdots}
\usepackage{varwidth}
\usepackage{pgf,tikz}
\usepackage{tikz-cd}
\usetikzlibrary{calc}
\usetikzlibrary{matrix,arrows}
\usepackage{stackrel}
\usepackage{mathrsfs}
\usepackage[shortlabels]{enumitem} 
\usepackage{stmaryrd} 
\usepackage{setspace} 
\spacing{1.15}

\newcommand\NN{\mathbb{N}}
\DeclareMathOperator\sgn{sgn}

\newcommand{\sch}{\operatorname{sch}}

\newcommand{\Ext}{\mathsf{Ext}}
\newcommand\inter[1]{[#1]}

\definecolor{newcol}{rgb}{0, 0.0, 0}
\DeclareTextFontCommand{\new}{\color{newcol}\bfseries\em}

\newcommand{\antishriek}{\text{\raisebox{\depth}{\textexclamdown}}}

\usepackage{amsthm}
\usepackage{thmtools}
\newtheoremstyle{mytheorem}
  {\topsep}   
  {\topsep}   
  {\itshape}  
  {0pt}       
  {\bfseries\color{newcol}} 
  {\color{newcol}}         
  {5pt plus 1pt minus 1pt} 
  {}          
\theoremstyle{mytheorem}
\newtheorem{theorem}{Theorem}[section]
\newtheorem{prop}{Proposition}[section]
\newtheorem{lema}{Lemma}[section]
\newtheorem*{conj*}{Conjecture}
\newtheorem{cor}{Corollary}[section]

\newtheoremstyle{mydefinition}
  {\topsep}   
  {\topsep}   
  {}  
  {0pt}       
  {\bfseries\color{newcol}} 
  {\color{newcol}}         
  {5pt plus 1pt minus 1pt} 
  {}          

\theoremstyle{mydefinition}

\newtheorem{deff}{Definition}[section]
\newtheorem{obs}{Observation}[section]

\newtheorem{remark}{Remark}[section]

\theoremstyle{plain} 
\newcommand{\thistheoremname}{}
\newtheorem{genericthm}[theorem]{\thistheoremname}

\newtheorem*{theorem*}{Theorem}

\usepackage{etoolbox}
\patchcmd{\section}{\normalfont}{\normalfont\large}{}{}

\definecolor{bluegray}{rgb}{0.4, 0.6, 0.8}
\usepackage[pdftex, colorlinks,bookmarks = true,bookmarksnumbered = true,linkcolor = black, citecolor = bluegray, urlcolor = bluegray ]{hyperref}
\usepackage{ytableau}
\usepackage{adjustbox}

\newcommand\spe[1]{\mathsf{#1}}
\newcommand{\e}{\spe{E}}
\newcommand{\x}{\spe{X}}
\newcommand{\y}{\spe{Y}}
\newcommand{\z}{\spe{Z}}

\newcommand{\cc}{\spe{C}}
\renewcommand{\L}{\spe{L}}

\renewcommand{\hom}{\operatorname{Hom}}

\makeatletter
    \renewcommand\subsection{%
        \@startsection{subsection}{4}{0 ex}%
           {1 em}%
           {-.5\baselineskip}%
           {\bfseries\normalfont\normalsize\bfseries}}
    \makeatother
    
\renewcommand{\tt}{\otimes}
\newcommand{\ps}{\mathbin\parallel}
\newcommand{\GSet}{\mathsf{Set}^\times}
\newcommand{\FSet}{\mathsf{FinSet}}
\newcommand{\Set}{\mathsf{Set}}

\newcommand{\Mod}{\mathsf{Mod}_\kk}
\newcommand{\Sp}{\mathsf{Sp}}

\newcommand{\CC}{{\mathsf{C}}}
\newcommand{\CX}{{C}}
\newcommand{\SC}{K}
\newcommand{\fs}{\sslash}
\newcommand{\bs}{\bbslash}

\renewcommand{\Sp}{\mathsf{Sp}}
\newcommand{\bb}[1]{\mathbf{#1}}

\DeclarePairedDelimiter\simp{\llbracket}{\rrbracket}

\newcommand{\Addresses}{{
  \bigskip
  \footnotesize

  \textsc{School of Mathematics, Trinity College, Dublin 2, Ireland}\par\nopagebreak
  \textit{E-mail address:} \texttt{pedro@maths.tcd.ie}
  }}
  
\usepackage{titletoc}

\titlecontents{chapter}
[0.2em] %
{\bigskip}
{\makebox[2em][r]{\thecontentslabel.}\hspace{0.333em}}
{\hspace*{-2em}}
{\hfill\contentspage}[\smallskip]

\titlecontents{section}
[.7 em]
{\small}
{\thecontentslabel.\hspace{3pt}}
{}
{\enspace\titlerule*[0.5pc]{.}\contentspage}
\titlecontents*{subsection}
[1.5 em]
{\footnotesize}
{\thecontentslabel. \hspace{3pt}}
{}
{}
[ --- \ ]
[]
\newcommand{\kk}{\Bbbk}
\usepackage{multicol}

\title{\setstretch{1.3}\textbf{The cohomology
of coalgebras in species}}
\author{
\textsc{Pedro Tamaroff}
}
\date{}
\begin{document}
\maketitle 

\begin{abstract}
Aguiar and Mahajan introduced a cohomology theory for the
twisted coalgebras of Joyal, with particular interest in 
the computation of their second cohomology group, which 
gives rise to their deformations. We use the Koszul duality 
theory between twisted algebras and coalgebras on the 
twisted coalgebra that gives rise to their cohomology 
theory to give a new alternative description of it
which, in particular, allows for its 
effective computation. We compute it completely in various 
examples, including those proposed by Aguiar and Mahajan, 
and obtain structural results: in particular, we study its
multiplicative structure and provide a Künneth formula.
\end{abstract}

 \thispagestyle{empty}

\section*{Introduction}

In their manuscript~\cite{Aguiar}*{Chapter 9},
Aguiar and Mahajan began investigating a
cohomology theory defined for the twisted
coalgebras of Joyal~\cite{Joyal,Joyal1986},
with the idea of producing certain deformations 
of them. Indeed,
as it often happens, a $2$-cocycle on a coalgebra
can be used to construct an infinitesimal
deformation of it and, conversely, infinitesimal
deformations of coalgebras correspond to
$2$-cocycles, as in~\cite{Gerstenhaber1964}. 
In particular, they defined the relevant complex
computing their cohomology theory, and performed
some low degree computations.

Although the deformations of Aguiar and Mahajan
are performed on coalgebras $\x$, which the theory
dictates should be studied by the Cartier cohomology 
groups~\cite{Cartier,Mastnak2009,Brzeziski2001}
 where we view $\x$ as a bicomodule
over itself, the fact their deformations only modify
coefficients imply simpler a cohomology theory
is needed to fulfill their desideratum: 
one can instead fix one
coalgebra, the exponential species $\e$, 
over which all coalgebras considered by 
Aguiar and Mahajan are bicomodules over,
and instead compute the Cartier cohomology groups
$H^*(\x,\e)$. With this at hand, the cohomology
group $H^2(\x,\e)$ is indeed in bijection
with the infinitesimal deformations Aguiar
and Mahajan seek to obtain.
These infinitesimal deformations always
integrate over a field of characteristic zero,
as we observe 
in Theorem~\ref{thm:integration},
although this statement is 
already implicit in~\cite{Aguiar}.

In this article, we use a version of
 Koszul duality theory~\cite{Priddy}
for algebras and coalgebras in combinatorial species
(called `twisted (co)algebras' in the literature~\cite{Joyal1986,VladimirMurray})
to provide a small Koszul complex that computes the
groups $H^*(\x,\e)$, the observation we make
is that $\e$ is a cocommutative cofree conilpotent
coalgebra and, as such is Koszul. To do this, 
we compute the Koszul dual algebra $\e^!$ corresponding
to $\e$, and show (Theorem~\ref{thm:mainresult})
that under mild homological conditions
on the species underlying $\x$, there is a cochain complex
\[K^*(\x) = \hom_{\mathsf{Sp}_\kk}(\x,\e^!)\]
which computes  $H^*(\x,\e)$, whose differential we compute
explicitly (Theorem~\ref{diffCC}). Having done this, we observe
that $\e$ is a bialgebra, which allows us to define an
internal tensor product $\otimes$ on $\e$-bicomodules, for which
we give a K\"unneth theorem (Theorem~\ref{thm:Kunneth}).
At the same time, we show how to define
products on $K^*(\x)$ coming from bicomodule codiagonal
maps \[\Delta : \x\longrightarrow \x\cdot \x\] which
make it into a dga algebra, whose formula we make 
explicit (Theorem~\ref{cupformula}).
In case $\x$ is a cosymmetric bicomodule (which
corresponds to the situation in which the coalgebra is
cocommutative) we show that the differential in
the Koszul complex vanishes, and that the cup product
is cocommutative.

Our results allows us to compute the
desired cohomology theory for many coalgebras considered
by Aguiar and Mahajan, which were unknown, and to
recover all the known computations that were done
at the time we obtained our results, such as those
of J. Coppola~\cite{tesisJavier}. We
compute all the cohomology groups of the coalgebra 
of linear orders considered, for example, in~\cite{Patras2004},
and by Aguiar and Mahajan (to the best of our knowledge,
they had computed its first two cohomology groups, and
found a non-trivial cocycle in $H^2(\L,\e)$). We do
the same for the coalgebras of partitions and compositions
and give partial computations for the species of graphs
with its cosymmetric structure given by graph restriction:
using the cup product in the Koszul complex, we find a polynomial
algebra generated by an element in degree four, thus
proving that the cohomology groups do not vanish in
infinitely many degrees.

\subsection*{Structure.}
This paper is organized as follows. In Section 1 we 
recall all the necessary ingredients needed to define the
cohomology theory of interest, which we do in Section 2.
In Section 3 we define and compute the Koszul complex
$K^*(\x)$ replacing ${C}^*(\x,\e)$, and prove Theorem~\ref{thm:mainresult}, Theorem~\ref{diffCC}, Theorem~\ref{thm:Kunneth}
and Theorem~\ref{cupformula}. In Section 4, we put our work in the
context of Koszul duality for (twisted) coalgebras and
compute the Koszul dual algebra to the coalgebra of linear orders.
 Finally, in Section 5, we briefly explain which kind of
 infinitesimal deformations the theory of~\cite{Aguiar}*{Section 9}
 defines, and observe that they can always be integrated.
\subsection*{Conventions.}
Throughout, $\kk$ is a unital
commutative ring, and when we
write $\tt$ y $\hom$, we will be
considering the usual functors on 
$\kk$-modules, unless stated otherwise.
We will write species with sans-serif
capital letters, and write sets with
capital italics.
A \emph{decomposition} $S$ of \emph{length} $q$ of a set $I$ is an ordered tuple 
$(S_1,\ldots,S_q)$ of possible empty subsets of $I$, which we call the 
\emph{blocks of $S$}, that are pairwise disjoint and whose union is $I$. We say 
$S$ is a \emph{composition} of $I$ if every block of $S$ is nonempty. It is 
clear that if $I$ has $n$ elements, every composition of $I$ has at most $n$ 
blocks. We will write $S\vdash I$ to mean that $S$ is a decomposition of $I$, 
and if necessary will write $S\vdash_q I$ to specify that the length of $S$ is 
$q$. Notice the empty set has exactly one composition which has length zero, the 
empty composition, and exactly one decomposition of each length $n\in 
\mathbb{N}_0$. If $T$ is a subset of $I$ and $\sigma : I \longrightarrow J$ is a 
bijection, we let $\sigma_T : T \longrightarrow \sigma(T)$ be the bijection 
induced by $\sigma$. 
 
\subsection*{Acknowledgments.} I
thank my MSc advisor M. Su\'arez--Alvarez for
the countless mathematical conversations we had
during my time in the University of Buenos Aires
(2013--2017) ---which I remember fondly---
and for his invaluable input during the
period of time when the thesis that this article 
is based on was written. I also thank anonymous
referees of a previous version of this manuscript
for pointing me to the references~\cite{Stover,Patras2004,DOI1981}.

\section{Algebras and coalgebras in species}

\subsection{The category of species.}

A 
\emph{combinatorial species} over a category $\mathcal C$ is a covariant functor $\x:\GSet 
\longrightarrow \mathcal C$, where $\GSet$ is the category of finite sets and bijections.
In particular, for every finite set~$I$ we have a map 
\[ \sigma\in 
\operatorname{Aut}(I) \longmapsto \x[\sigma]\in \operatorname{Aut}(\x[I])\] which 
gives an action of the symmetric group with letters in $I$ on $\x[I]$. The 
category $\GSet$ is a grupoid, and it has as skeleton the full subcategory 
spanned by the sets 
$[n]=\{1,\ldots,n\}$ (in particular, $[0]=\varnothing$), and a species is determined, up to isomorphism, by 
declaring its values on the finite sets $[n]$ and on every $\sigma \in S_n$. 
We denote by $\mathsf{Sp}(\mathcal C)$ the category $\operatorname{Fun}(\GSet,\mathcal C)$ of 
species over $\CC$, whose morphisms are natural transformations. 

{{}} Our main interest will lie on species over sets or vector spaces. We write 
$\mathsf{ Sp}$ for the category of species over $\Set$, the category of sets and 
functions, and call its objects \emph{set species}. If a species takes values on 
the subcategory $\FSet$ of finite sets we call it a \emph{finite} set species, 
and if $\x(\varnothing)$ is a singleton, we say it is \emph{connected}. We write 
$\mathsf{Sp}_\kk$ for the category of species over $\Mod$, the 
category of modules over $\kk$, and call its objects \emph{linear species}.  If a 
species takes values on the subcategory of finite generated 
modules we call it a linear species of \emph{finite type}, and we say it is 
\emph{connected} if $\x(\varnothing)$ is $\kk$-free of rank one.

{{}} Denote by $\kk[-]$ the functor $\Set \longrightarrow \Mod$ that 
sends a set $\x$ to the free $\kk$-module with basis $\x$, which we will denote by 
$\kk\x$, and call it the \emph{linearization of $\x$}. By postcomposition, we obtain 
a functor $ L:\mathsf{Sp} \longrightarrow \mathsf{Sp}_k$ that sends a 
set species $\x$ to the linear species $\kk\x$. The species in $\mathsf{Sp}_\kk$ that 
are in the image of $\kk[-]$ are called \emph{linearized species}.  
 
\subsection{The Cauchy product.}
  
Let $\x$ and $\y$ be linear species over $\kk$. 
The Cauchy product $\x\cdot \y$ is the 
linear species such that for every finite set $I$
\[ (\x\cdot \y)[I] = \bigoplus\limits_{(S,T)\vdash I} \; \x[S]\tt \y[T]\]
the direct sum running through all decompositions of $I$ of length two.
This construction extends to produce a bifunctor 
$-\cdot -: \Sp_\kk
\times \Sp_\kk \longrightarrow \Sp_\kk$. In what follows, whenever we speak of the 
category $\Sp_\kk$, we will view it as a monoidal category with the monoidal 
structure.

It is important to notice the construction of the Cauchy product in $\Sp_\kk$ 
carries over to the category $\Sp(\mathcal C)$ when $\mathcal C$ is any monoidal category~\cite{Kassel} with 
finite coproducts which commute with its tensor product. The main example of 
this phenomenon happens when $\mathcal C$ is the category $\Set$. If $\x$ and $\y$ are 
set species, the species $\x\cdot \y$ has
\[ (\x\cdot \y)[I]= \bigsqcup_{(S,T)\vdash I} \x[S]\times \y[T],\]
so that a structure $z$ of species $\x\cdot \y$ over a set $I$ is determined by a 
decomposition $(S,T)$ of~$I$ and a pair of structures $(z_1,z_2)$ of species $\x$ 
and $\y$ over $S$ and $T$, respectively.

\subsection{Coalgebras and bialgebras in species.}

An associative algebra $(\x,\mu,\eta)$ in the category $\Sp_k$ 
of $k$-linear species will simply be called an \emph{algebra};
if we need to make a clear distinction between usual algebras
in $ \Mod$ and those in species, we will use the term
`twisted algebra' as in~\cite{VladimirMurray}*{Chapter 4}.
Such an object is determined by a product 
$ \mu :\x\cdot \x\longrightarrow \x$ and a unit $\eta : \bb{1} 
\longrightarrow \x$. Specifying the first amounts to giving its components 
\[\mu(S,T) :\x[S]\tt\x[T] \longrightarrow\x[I]\] at each decomposition $(S,T)$ of 
every finite set $I$, and specifying the latter amounts to a choice of the 
element $\eta(\varnothing)(1) \in \x(\varnothing)$, which we will denote by $1$ 
if no confusion should arise. 
We think of the product as an operation that glues partial structures 
on~$I$, and of the unit as an ``empty'' structure. 
Dually, a coassociative coalgebra $(\x,\Delta,\varepsilon)$
in $\Sp_\kk$, which we call simply a \emph{coalgebra},
is determined by a coproduct $\Delta : \x\longrightarrow
\x\otimes \x$ and a counit $\varepsilon : \x \longrightarrow \bb{1}$. 
The coproduct has, at each 
decomposition $(S,T)$ of $I$, a component 
\[\Delta(S,T):\x[I]\longrightarrow 
\x[S]\otimes\x[T],\] which we think of as
 breaking up a combinatorial structure on 
$I$ into substructures on $S$ and $T$, 
while the counit is a map of $\kk$-modules 
$\x[\varnothing] \longrightarrow \kk$. 
As usual, a bialgebra is an algebra in the category
of coalgebras in species or, equivalently, 
a coalgebra in the category of algebras in species.

We can also consider `set-theoretic' analogues of 
the three objects above.  For example, by a set
algebra we mean a set species $\x_0$ endowed
with product maps \[\mu_{S,T} :\x_0[S]\times
\x_0[T]\longrightarrow \x_0[S\sqcup T],\]
one for each pair $(S,T)$ of (disjoint) finite
sets, satisfying suitable
associativity conditions. Similarly, by a set
coalgebra we mean a set species  $\x_0$ 
endowed with product and coproduct maps
\[\Delta_{S,T} : \x_0[S\sqcup T]
	\longrightarrow \x_0[S]\times \x_0[T]\] 
satisfying suitable coassociativity conditions.

\subsection{The exponential species.}
We write $\e$ for the $\kk$-linear species
for which $\e[I]$ is one dimensional with basis
the singleton $\{I\}$. For convenience, we 
write $e_I$ for the corresponding basis element.
We call this the \emph{exponential species},
and observe that, by definition,
it is a linearized species of the form
$\kk \e_0$ for the set species $\e_0$
with $\e_0[I] =\{I\}$ for each finite set $I$.
The following proposition endows the exponential
species with a bialgebra structure.

\begin{prop}
The linearized exponential species $\e$ is a bialgebra with 
product and coproduct with components 
\[ \mu(S,T) : \e[S]\tt \e[T] \longrightarrow \e[I], \qquad \Delta(S,T) : 
\e[I]\longrightarrow \e[S]\tt \e[T]\]
at each decomposition $(S,T)$ of a 
finite set $I$ such that 
$\mu(S,T)(e_S\tt e_T) =  e_{I}$ and 
$\Delta(S,T)(e_I)=e_S\tt e_T$
and with unit and counit the morphisms 
$\varepsilon: \e \longrightarrow \bb{1}$ 
and $\eta : \bb{1} \longrightarrow \e$ 
such that $\varepsilon(e_\varnothing) = 
1$ and $\eta(1) = e_\varnothing$. 
\end{prop}
\begin{proof}
The verifications needed to prove this follow immediately from the fact that 
$\e_0[I]$ is a singleton for every finite set $I$. 
\end{proof}

Although Aguiar and Mahajan are interested
in coalgebras in species, the following proposition
shows that, when such coalgebras arise from set
coalgebras, they automatically become bicomodules
over the exponential species. This 
allows us to work with a linear category of 
bicomodules, as opposed to a category of
coalgebras, and to interpret Aguiar and
Mahajan's cohomology theory as a derived functor
in the category of $\e$-bicomodules.

\begin{prop}\label{prop:terminalE} 
The exponential species $\e_0$ admits a unique structure of set-theoretic bialgebra so that if
$\,\x_0$ is a set-theoretic coalgebra in $\Sp$, the linearization of the unique morphism of species $\x_0\longrightarrow \e_0$ is a morphism of  coalgebras.
In particular, every coalgebra coming from a set-theoretic
coalgebra is canonically an $\e$-bicomodule. 
 
\end{prop}

\begin{proof}
If $S$ is a singleton set and $X$ is any set, there is a unique function 
$X\longrightarrow S$, and it follows from this, first, that the bialgebras 
structure defined on $\e$ is the only linearized bialgebra structure, and, second, 
that if $\x$ is a species in $\Sp$, there is a unique morphism of species 
$\x\longrightarrow \e$. If $\x$ is a pre-coalgebra in $\Sp$, the following square 
commutes because $\e[S]\times \e[T]$ has one element:
$$
\begin{tikzcd} 
\e[I]
	\arrow{r}{\Delta}
& \e[S]\times \e[T] 
	\\
\x[I]
	\arrow{u}
	\arrow{r}{\Delta}
& \x[S]\times\x[T],
	\arrow[swap]{u}
\end{tikzcd}
$$
and, by the same reason, $\x\longrightarrow \e$ is pre-counital. All this shows 
that the exponential species $\e$ is terminal in the category of linearized 
coalgebras. This completes the proof.
\end{proof}

\subsection{Representations of the exponential species.}
We will fix some useful notation to deal 
with coalgebras. Let $\x=\kk\x_0$ be a 
linearized species that is a coalgebra 
in $\Sp_\kk$. If $z$ is an element of $\x_0[I]$, 
we write 
\begin{equation}
\Delta(I)(z) = 
\sum z\bs S\otimes z\fs T 
\end{equation} with $z\bs S \tt z\fs 
T$ denoting an element of $\x[S] \tt\x[T]$  
\`a la Sweedler, not necessarily an elementary tensor.
Consider now a left $\e$-comodule $\x$ with coaction $\lambda : \x\longrightarrow 
\e\cdot \x$. Since $\e[S]= \kk e_S$, the component $\x[I]\longrightarrow \e[S]\tt 
\x[T]$ can canonically be viewed as map $\x[I]\longrightarrow\x[T]$ which we 
denote by $\lambda_T^I$, and call the it \emph{the restriction from $I$ to $T$ 
to the right}. 

In these terms, that $\lambda$ be counital means $\lambda_I^I$ is the identity 
for all finite sets $I$, and the equality $1\tt \lambda\circ \lambda = \Delta\tt 
1 \circ \lambda$, which expresses the coassociativity of $\lambda$, translates 
to the condition that we have $\lambda_A^I = \lambda_A^B\circ \lambda_B^I$ for 
any chain of finite sets $A\subseteq B\subseteq I$. Let
us write
$\FSet^{\mathrm {inc}}$ for the category of finite sets and inclusions. This discussion leads to the
following:

\begin{prop}
The category of left 
$\e$-comodules  in $\Sp_\kk$ is
equivalent to the category of
pre-sheaves \[ 
\FSet^{\mathrm {inc}}\longrightarrow  \Mod.\]

\end{prop}
These are usually called FI-modules in the
literature, see for example~\cite{FIMODS}. 
 When convenient, we will write 
$z\fs S$ for $\lambda_S^I(z)$ without explicit mention to $I$, which will 
usually be understood from context. Using this notation, we can write the 
coaction on $\x$ as
\[ \lambda(I)(z) = \sum e_S \tt z\fs T.\]

{{}}
Of course the same consideration apply to a right $\e$-comodule, and we write 
$z\bs T$ for $\rho_T^I(z)$. If $\x$ is both a left and a right $\e$-comodule with 
coactions $\lambda$ and $\rho$, the compatibility condition for it to be an 
$\e$-bicomodule is that, for any finite set $I$ and pair of non-necessarily 
disjoint subsets $S,T$ of $I$, we have $\rho_{S\cap T}^{S} \lambda^I_{S} = 
\lambda_{S\cap T}^T \rho_T^I$. 
In a completely analogous fashion, 
there is a category $\FSet^{\mathrm {binc}}$,
and the category of $\e$-bicomodules is 
equivalent to the category of pre-sheaves \[\FSet^{\mathrm {binc}} \longrightarrow \Sp_\kk.\]
We leave its construction to the categorically inclined reader.
There is a close relation between linearized coalgebras and linearized 
$\e$-bicomodules, as described in the following proposition. In fact, all the bicomodules we
are interested in arise from the construction 
in this proposition.

\begin{prop}
Let $(\z,\Delta)$ be a linearized coalgebra, and let $f_\z : \z\longrightarrow \e$ be 
the unique morphism of linearized coalgebras described in 
Proposition~\ref{prop:terminalE}. There is on $\z$ an $\e$-bicomodule structure so
that the coactions $\lambda : \z\longrightarrow \e\cdot \z$ and $\rho : 
\z\longrightarrow \z\cdot \e$ are obtained from postcomposition of $\Delta$ with 
$f_\z\tt 1$ and $1\tt f_\z$, respectively. \qed
\end{prop}

We refer the reader to \cite{Aguiar}*{Chapter 8, \S 3, Proposition 29}. Remark 
that, with this proposition at hand, the notation introduced for bicomodules and 
that introduced for coalgebras is consistent.

 	\section{The cohomology of coalgebras}\label{ch:thecohomology}
\subsection{Definitions and first examples.}
Let $\cc$ be a coalgebra and $\x$ a $\cc$-bicomodule, and let us define a
cosimplicial $\kk$-module
$\CX^*(\x,\cc)$ as follows.
For each $n\in\NN$: 
\begin{tenumerate}
\item Define $\CX^n(\x,\cc)$ to
be $\hom_{\mathsf{Sp}_k}(\x,\cc^{\cdot n} )$
 the set of maps in  $\mathsf{Sp}_k$
 from $\x$ to the iterated tensor
 product $\cc^{\cdot n}$.
 \item For $0<i<n+1$, consider the map
 $d^i : 
 \CX^n(\x,\cc) \longrightarrow 
 \CX^{n+1}(\x,\cc)$
 induced by post-composition with the
 coproduct of $\cc$ at the $i$th
 position. 
 \item For $i=0,n+1$, let
 $d^0,d^{n+1} : 
 \CX^n(\x,\cc) \longrightarrow 
 \CX^{n+1}(\x,\cc)$ 
 be the maps obtained post composing
 with the left and right 
 comodule maps of $\x$, respectively. 
\end{tenumerate}

It is straightforward to check that,
by virtue of the coassociativity of
$\cc$ and the bicomodule axioms,
the maps above satisfy the usual 
cosimplicial identities. It follows
that if for each $n\in \NN$ we 
define the alternated sum $\delta^n=\sum_{i=0}^{n+1} (-1)^i d^i$, we
obtain a cohomologically graded complex 
$(\CX^*(\x,\cc),\delta^*)$. 

\begin{deff}
The \emph{cohomology of $\x$ with values in $\cc$} is the cohomology
of the cochain complex $(\CX^*(\x,\cc),\delta^*)$, and we denote it by $H^*(\x,\cc)$. 
\end{deff}

The homologically inclined reader will
notice that these cohomology
groups are equal to the groups $\Ext^*(\x,\cc)
 $
with the $\Ext$ taken in the category of $\cc$-bicomodules~\cite{DOI1981}.
In the following we will mainly consider the case in which 
$\cc$ is the exponential species, but will make it clear 
when a certain result can be extended to other 
coalgebras. Usually, it will be necessary that $\cc$ is 
linearized and with a linearized bialgebra structure, and we 
will usually require that $\cc$ be connected. Because of
the plethora of relevant examples of such bialgebras found 
in~\cite{Aguiar} and other articles by the same authors, such 
as~\cite{AguiarHopf}, there is no harm in restricting ourselves
to such species.

\subsection{The theory for the exponential species.}
Fix an $\e$-bicomodule $\x$. The complex $\CX^*(\x,\e)$, which 
we will denote more simply by $\CX^*(\x)$, has in degree $q$ the 
collection of morphisms of species $\alpha : \x\longrightarrow \e^{\cdot q}$. Such a 
morphism is determined by a collection of $k$-linear maps $\alpha(I) : 
\x[I]\longrightarrow \e^{\cdot q}[I]$, one for each finite set $I$, which is 
equivariant, in the sense that for each bijection $\sigma : I\longrightarrow J$ 
between finite sets, and every $z\in\x[I]$, the equality $\sigma(\alpha(I)(z)) = 
\alpha(J)(\sigma z)$ holds. More generally:

\begin{prop} If 
$\x,\y_1,\ldots,\y_r$ are linear species, a map of species $\alpha : 
\x\longrightarrow \y_1 \cdot \ldots\cdot \y_r$ determines and is determined by 
a choice of equivariant $\kk$-module maps 
\[ \alpha(I):\x[I]\longrightarrow \bigoplus_{(S_1,\ldots,S_r)} \y_1(S_1)\otimes\cdots\otimes 
\y_r(S_r),\]
one for each finite set $I$, where the direct sum runs through decompositions 
$(S_1,\ldots,S_r)$ of length $r$~of the finite set $I$. \qed
\end{prop}

{{}} The map $\alpha(I)$ is 
specified uniquely by its \emph{components} at each decomposition 
$S=(S_1,\ldots,S_r)$, which we denote $\alpha(S_1,\ldots,S_r)$ without further 
mention to the set $I$ which is implicit. Moreover, 
it suffices to specify $\alpha_I$ for $I$ the sets $\simp{n}$ with $n\in \mathbb 
N_0$.  This said, we will usually define a map $\alpha : \x \longrightarrow \y_1 
\cdot \ldots \cdot \y_r$ by specifying its components at each decomposition of $I$ of length~$r$.   

For each finite
set $I$, the space $\e^{\cdot q}[I]$ is a free $\kk$-module with basis the tensors of the form 
$ F_1 \tt \cdots \tt F_q$ with $F=(F_1,\ldots,F_q)$ a decomposition of $I$; for 
simplicity, we use the latter notation for such basis elements. In terms of this 
basis, we can write
\[ \alpha(I)(z) = \sum_{F\vdash_q I} \alpha(F)(z)\, F \]
where $\alpha(F)(z)\in \kk$. 
Observe that the equivariance condition says that, for a bijection 
$\sigma : I\longrightarrow J$, and $(F_1,\ldots,F_q)$ a decomposition of $I$, we 
have
\[ \alpha(F_1,\ldots,F_q)(z) = \alpha(\sigma(F_1),\ldots,\sigma(F_q))(\sigma 
z)\]
for each $z\in\x[I]$.
Now fix a $q$-cochain $\alpha : \x\longrightarrow \e^{\cdot q}$ in $\CX^*(\x)$. By 
the remarks in the last paragraph, to determine the $(q+1)$-cochain $\delta\alpha : 
\x\longrightarrow \e^{\cdot (q+1)}$ it is enough to determine its components. 

\begin{lema} For each 
decomposition $F=(F_0,\ldots,F_q)$ of a set $I$, then the component of 
the $i$th coface $d_i \alpha$ at $F$ is given, for $z\in\x[I]$, by
\begin{equation}
\label{eq:cofacemaps}
(d^i \alpha)(F_0,\ldots,F_q)(z)=
\begin{cases}
 \alpha(F_1,\ldots,F_q)(z\fs F_0^c) & \text{ if $i=0$,} \\
\alpha(F_0,\ldots,F_i\cup F_{i+1},\ldots,F_{q+1})(z) & \text{ if $\;0<i<q+1$,} \\
\alpha(F_0,\ldots,F_{q-1})(z\bs F_q^c) & \text{ if  $i=q+1$}.
\end{cases}
\end{equation}
\end{lema}
\begin{proof}
Indeed, let us follow the
prescription above  
and compute each coface map explicitly. If $z\in\x[I]$, to compute 
$d^0\alpha(z)$, we must coact on $z$ to the left and evaluate the result 
at $\alpha$, that is
\[ (1\tt \alpha\circ \lambda)(I)(z) = \sum_{(S,T)\vdash I} e_S \otimes 
\alpha(T)(z\fs T), \] and the coefficient at a decomposition 
$F=(F_0,\ldots,F_q)$ is $\alpha(F_1,\ldots,F_q)(z\fs F_0^c)$. The same argument 
gives the last coface map. Now consider $0<i<q+1$, so that we must take $z\in 
\x[I]$, apply $\alpha$, and then comultiply the result at coordinate $i$. 
Concretely, write \[ \alpha(I)(z) = \sum_{F\vdash_q I} \alpha(I)(F)(z) F\] and 
pick a decomposition $F' = (F_0,\ldots,F_q)$ into $q+1$ blocks of $I$. There 
exists then a unique $F\vdash_q I$ such that $1^{i-1}\tt \Delta \tt 1^{q-i}(F) = 
F'$, to wit, $F=(F_0,\ldots,F_i\cup F_{i+1},\ldots,F_q)$, and in this way we 
obtain the formulas of Equation~\eqref{eq:cofacemaps}. 
\end{proof}

Since $\e$ is counital, the complex
above admits codegeneracy maps, 
which are much
 easier to describe: they are obtained by inserting an 
empty block into a decomposition. Concretely, for each $j\in \{ 0,\ldots,q+1\}$, 
\[(\sigma^j\alpha)(F_1,\ldots,F_q)(z) = 
\alpha(F_1,\ldots,F_j,\varnothing,F_{j+1},\ldots,F_q).\] As a consequence of 
this, a cochain $\alpha : \x \longrightarrow \e^{\cdot q}$ in $\CX^*(\x)$ is in the 
normalized subcomplex $\overline{\CX}^*(\x)$ if its components are such that 
$\alpha(F)(z)=0\in k$ whenever $F$ contains an empty block. Alternatively, we 
can construct a (non-unital) coalgebra $\overline \e$ with  $\overline{\e}(\varnothing)=0$ and 
$\overline{\e}[I]=\e[I]$ whenever $I$ is nonempty, and describe the normalized complex 
$\overline{\CX}^*(\x)$ as the complex of maps $\x\longrightarrow \overline{\e}^{\cdot *}$ 
with differential induced by the alternating sum of the coface maps we just 
described. 

\begin{remark} For each
finite set $I$ the space
$\overline{\e}^{\cdot q}[I]$ has basis the \emph{compositions} of 
$I$ into $q$ blocks, while $\e^{\cdot q}[I]$ has basis the \emph{decompositions} of 
$I$ into $q$ blocks. In particular, $\overline{\e}^{\cdot q}[I]=0$ if $q>\# I$, while 
$\e^{\cdot q}[I]$ is always nonzero. This observation will be useful in 
Section~\ref{ch:theCC}.
\end{remark}

\subsection{The cobar complex and cup products.}\label{sec:thecup}  

Since the coalgebra $\e$ is,
in fact, a cocommutative Hopf algebra, we
can endow the complex $\CX^*(\x)$
with the structure of a dga algebra
and hence produce on the cohomology
groups $H^*(\x,\e)$ a structure
of an associative algebra,
as follows.

{{}} 
First, let us give an alternative
way of constructing the complex
$\CX^*(\x)$. Let $\Omega^*(\e)$ denote
the cobar construction on the coalgebra
$\e$ (see ~\cite{LV}*{Chapter 3} for the
case of usual associative algebras). This is a dg algebra
which is freely generated by $s^{-1}\overline{\e}$, 
the shift of the species $\e$ without the counit, and
whose differential is induced from
the coproduct of $\e$: it is the unique
coderivation extending the map
\[\Delta : s^{-1}\overline{\e}
 	\longrightarrow  (s^{-1}
 		\overline{\e})^{\cdot 2}
 		\subseteq \Omega^*(\e).\]
We can then form the space
\[ \hom_{\Sp_\kk}(\x,\Omega^*(\e)) \]
which, as a graded vector space,
coincides with the normalized 
complex for $\CX^*(\x,\e)$: the
way we shifted $\e$ makes sure that
maps $\x\longrightarrow \overline{\e}^{\tt n}$ live in degree $n$. Observe,
moreover, that this hom-set above
inherits a differential $\delta_1$ by
post-composition with the differential
\[ d: \Omega^*(\e) 
	\longrightarrow \Omega^{*+1}(\e), \]
and this coincides in fact with the
internal sum of the coface maps above,
omitting the endpoints $0$ and $n+1$. 
To obtain the full differential $\delta$, 
we consider the canonical degree $-1$ injection
$\tau : \e\longrightarrow \Omega^*(\e)$  
and the differential $\delta_2$ 
obtained by the following
composition where $p$ is the
degree of $\varphi: \x \longrightarrow 
\e^{\otimes p}$:
\[ \delta_2(\varphi) =  
\mu_{\Omega^*(\e)}(\varphi\otimes\tau\circ\lambda
	 +(-1)^{p} \tau\otimes\varphi\circ\rho). \]
A perhaps tedious but straightforward
computation shows that $\delta_1-\delta_2$
coincides with $\delta$, so that we obtain a 
new description of the complex $\CX^*(\x,\e)$ 
as a complex `twisted' by $\tau$ (the summand 
$\delta_2$ is the twist determined by $\tau$):
\[ \CX^*(\x,\e) = 
	\left(\hom_\tau(\x,\Omega^*(\e)), 
			\delta_1-\delta_2\right). \] 

\begin{prop}
The dg coalgebra $\Omega^*(\e)$ is in fact
a dg bialgebra if we endow it with
the \emph{shuffle product} induced
from the cocommutative coproduct of
$\e$, which we will denote by
$\Delta_{\Omega^*(\e)}$.
\end{prop}

\begin{proof}
This statement is completely dual
to the classical statement (see for
example Chapter 8 in~\cite{McCleary}) that if
$A$ is a commutative algebra then 
the bar construction $BA$ is a 
commutative algebra with the
shuffle product induced from
the commutative product of $A$. 
We remind the reader that it is
crucial that $A$ be commutative
(and hence, in our case, that 
$\e$ be cocommutative) for
this product to be compatible with
the differential of $BA$. 
\end{proof}

\begin{deff}
  We define the external product
   \[ -\times - : \CX^*(\x,\e)\otimes
   	\CX^*(\x,\e)  \longrightarrow 
   		\CX^*(\x\cdot \y,\e)\]
   	so that for two cochains
$\varphi,\psi\in \CX^*(\x,\e) $  we have
$\varphi\times \psi = 
\mu_{\Omega^*(\e)}\circ(\varphi\otimes\psi)$. 
\end{deff}

{{}} 
Note that we use the fact $\e$ is a
Hopf algebra, which implies that the category of $\e$-bicomodules
admits an internal tensor product. Concretely,
if $\x$ and $\y$ are $\e$-bicomodules, we endow the
tensor product $\x\otimes \y$ with the left and right
diagonal actions coming from the product of $\e$.
In case we have a coproduct
 map $\Delta : \x \longrightarrow 
\x \otimes \x$
making $\x$ into a coalgebra in the category of $\e$-bicomodules, we can use
this external product to obtain a cup product
in $\CX^*(\x)$, which we will write 
 \[ -\smile - : \CX^*(\x,\e)\otimes
   	\CX^*(\x,\e)  \longrightarrow 
 		   		\CX^*(\x,\e).\]

\begin{remark}
In general,
the algebra~$H^*(\x)$ will be non-commutative: for example, if $\x$
is concentrated in cardinal zero, then the datum of~$\x$ really amounts to
that of the coalgebra~$\x[\varnothing]$, a coalgebra in~$\kk$-modules, and
$H^*(\x)$ is the algebra dual to it, which may very well be
non-commutative.
\end{remark}

If $\x$ is an $\e$-bicomodule and $\Delta:\x\to\x\cdot\x$ a
morphism of $\e$-bicomodules, we write, for each $I$ and each
$z\in\x[I]$,
  \begin{equation}
  \Delta[I](z) = \sum_{(S,T)\vdash I} z_{(S)}\otimes z^{(T)}
  \end{equation}
\`a la Sweedler, with each summand $z_{(S)}\otimes z^{(T)}$ appearing here
standing for an element ---not necessarily an elementary tensor--- of the
submodule $\x[S]\otimes\x[T]$ of $(\x\cdot\x)[I]$,
as in Equation (1). 
If~$\alpha:\x\to\e^{\cdot p}$ and $\beta:\x\to\e^{\cdot q}$ are a
$p$- and a $q$-cochain in the complex~$C^*(\x)$, then their product
$\alpha\smile\beta\in C^{p+q}(\x)$ has coefficients given by
  \[
  (\alpha\smile\beta)(F)(z) 
    = 
      \alpha(F_{1,p})(z_{(F_{1,p})}) 
      \cdot 
      \beta(F_{p+1,p+q})(z^{(F_{p+1,p+q})})
\]
for all $I$, all decompositions $F=(F_1,\dots,F_{p+q})$ of~$I$ and all
$z\in\x[I]$. Here we are being succinct and writing
$F_{i,i+j}$ for both the decomposition $(F_i,\ldots,F_{i+j})$ obtained from $F$ and for the union of this
decomposition. 
Our main source of examples of coalgebras in $\e$-bicomodules comes from the
following simple observation:

\begin{prop}\label{prop:coalg}
Let $\x_0$ be a nonempty set-valued species with left and right restrictions and let $\x$ be the $\e$-bicomodule
obtained by linearization from~$\x_0$. There is a morphism of
$\e$-bicomodules $\Delta:\x\to\x\otimes\x$ such that
 $
  \Delta[I](z) = \sum_{(S,T)\vdash I}z \bs S\otimes z\fs T
  $
for each finite set $I$ and each $z\in\x_0[I]$. \qed
\end{prop}

In what follows, we will usually consider every $\e$-bicomodule whose underlying species is a linearization of a coalgebra in the way described in this proposition.
In particular, in the notation we used in Equation
(1) for coproducts of coalgebras and the notation
we used for bicomodules in Equation (2), the
term $z \bs S\otimes z\fs T$ corresponds to $z_{(S)}\otimes z^{(T)}$. 

 \section{An alternative description of cohomology}\label{ch:theCC}
The objective of this chapter is to obtain an alternative and more useful 
description of the cohomology groups of an $\e$-bicomodule $\x$. We show that
$\e$ is a Koszul coalgebra, compute its Koszul dual algebra and show that 
if $\x$ is \emph{weakly projective}, that is, if for each non-negative integer $j$, the component $\x[j]$ 
is a projective $\kk S_j$-module, there is a Koszul complex that calculates 
$H^*(\x)$, and which can be used for effective computations. 

\subsection{Koszul duality for coalgebras in species.}

Let us now consider the weight grading in $\e$ where we
put $\e(n)$ in weight $n$. Then the algebra $\Omega^*(\e)$
admits a homological grading where $(F_1,\ldots,F_q)$ is
in cohomological degree $q-\sum_{i=1}^q |F_i|$. With this
grading at hand, the differential in $\Omega^*(\e)$ is of
degree $-1$. In particular, $H^0(\Omega^*(\e))$ is a quotient
of $\Omega^*(\e)$. 

\begin{deff}
The \emph{Koszul dual algebra} $\e^!$ is the zeroth homology group
$H^0(\Omega^*(\e))$.  
\end{deff}

For $j\geqslant 1$ and for each integer 
$p\geqslant -1$, let $\Sigma_p(j)$ be 
the collection of compositions of length 
$p+2$ of~$[j]$. We will identify the 
elements of $\Sigma_{j-2}(j)$ with permutations 
of $[j]$ in the obvious way. 
There are face maps $\partial_i : 
\Sigma_p(j) \longrightarrow \Sigma_{p-1}(j)$ 
for $i\in \{0,\ldots,p\}$ given by
\[ \partial_i (F_0,\ldots,F_i, F_{i+1},\ldots,F_{p+1})=
 (F_0,\ldots,F_i\cup F_{i+1},\ldots,F_{p+1})\]
that make the sequence of sets $\Sigma_*(j) 
= (\Sigma_p(j))_{p\geqslant -1}$ 
into an augmented semisimplicial set. 
We write $\kk\Sigma_*(j)$ for the augmented 
semisimplicial $k$-module obtained by 
linearizing $\Sigma_*(j)$, and $\kk\Sigma_*(j)'$ 
for the dual semicosimplicial augmented $\kk$-module.

There is an action of $S_j$ on each $\Sigma_p(j)$ by permutation, so that if 
$\tau \in S_j$ and if $(F_0,\ldots,F_t)$ is a composition of $[j]$, then
\[ \tau(F_0,\ldots,F_t) = (\tau(F_0),\ldots,\tau(F_t)).\]
It is straightforward to check the coface maps are equivariant with respect to 
this action, so $\Sigma_*(j)$ is, in fact, an augmented semisimplicial 
$S_j$-set. Consequently, $\kk\Sigma_*(j)$ and $\kk\Sigma_*(j)'$ have corresponding 
$S_j$-actions compatible with their semi(co)simplicial structures.
This complex $\Sigma_*(j)$ is known in the literature as the \emph{Coxeter complex for 
the braid arrangement}, and its cohomology can be completely described. 
 We refer the reader to \cite{AMCox} and 
\cite{Brown} for details. 

\begin{prop}\label{coxeter} The complex associated to $\kk\Sigma_*(j)'$ computes 
the reduced cohomology of a $(j-2)$-sphere with coefficients in $\kk$, that is, 
\[
H^p(\kk\Sigma(j)')=
	\begin{cases} 0 & \text{ if $p\neq j-2$} \\
	\kk \llbracket \xi_j \rrbracket & \text{ if $p=j-2$ } \end{cases}\]
The non-trivial term is the 
$\kk$-module freely generated by the class of the map 
$ \xi_j : 
\kk\Sigma_*(j)\longrightarrow k$  such that $\xi_j(\sigma) = \delta_{\sigma=1}$
and the action of $\kk S_j$ on $H^{j-2}(\kk\Sigma_*(j)')$ is the sign representation. \qed
\end{prop}

\begin{remark}
In what follows, $\sgn_j$ will denote the sign representation of $\kk S_j$ just 
described. Note that, when $j=1$, $S^{j-2}=\varnothing$, and the reduced 
cohomology of such space is concentrated in degree $-1$, where it has value $\kk$. 
\end{remark}

We now observe that $\Omega^*(\e)$ is equivariantly isomorphic to 
the normalization of $\Sigma_{*-2}$, which immediately implies
the following result:

\begin{theorem}
The projection $\Omega(\e)\longrightarrow \e^!$ 
is a quasi-isomorphism of algebras.\qed
\end{theorem}

This theorem also allows us to identify $\e^!$ as an algebra.
The signed exponential species~\cite{Aguiar}*{Section 9.3.1} is given by $\e^-[I] = \bigwedge^{|I|} \kk^I$. 
As such, it is one dimensional, and the action of
$S_I$ on it is the sign representation. We define
the a desuspension functor on symmetric sequences
by $(\Sigma^{-1}\x)[I] = s^{-|I|}\x[I]$. It is clear how to extend
this to a functor on (co)algebras. 

\begin{cor}
 The algebra $\e^!$ is isomorphic to the desuspension
 $\Sigma^{-1}\e^-$.\qed
 \end{cor}
 
 One can in fact show that $\e$ is Koszul without resorting
 to the above geometrical considerations. Indeed, $\e$ is
 the cofree conilpotent cocommutative coalgebra over the unit
 species (that is, a single generator in degree zero)
 and, as such, it is immediately Koszul. Moreover, its
 Koszul dual is the commutative algebra over
 the suspension of the unit species (that is, a single
 generator in degree minus one). This is exactly the species
 $\Sigma^{-1}\e^-$ above. Below, it will be important to note
 that $\e^-$ carries signs in its product: for example, if
 $e_1$ and $e_2$ are generators corresponding to the singletons
 $\{1\}$ and $\{2\}$, then $e_{\{1,2\}}=e_1e_2 = - e_2e_1$.
  
 \begin{remark} The previous paragraph
  shows us how the purely algebraic theory of
Koszul duality can shed light into combinatorics: 
the observation above implies immediately
that the Coxeter complex has the homology of a sphere,
for example, and that the representation of this
top homology group is the sign representation without
doing \emph{any} computation at all.
\end{remark}

All the work done so far shows that the complex 
$\hom(\x,\Omega(\e))$ computes the cohomology of~$\x$. 
If we further assume that~$\x$ is weakly projective, then
the functor $\hom(\x,-)$ is exact so that it preserves
quasi-isomorphisms. We deduce that:

\begin{theorem}\label{thm:mainresult}
If $\x$ is weakly projective, the Koszul complex
$K^*(\x) = \hom(\x,\e^!)$ computes the cohomology
groups $H^*(\x)$. Explicitly, we have that for each
$p\in\NN$,
\[ K^p(\x) = \hom_{S_p}(\x[p],\sgn_p) \]
and the differential is given by
\[ d(\varphi) = [\tau,\varphi] = \tau \frown \varphi  -(-1)^{|\varphi|} \varphi \frown \tau \]
where $\tau : \e \longrightarrow \e^!$ is the composition
$\e \longrightarrow \Omega(\e) \longrightarrow \e^!$
and $\tau\frown \varphi$ is given by the composition
\[ \x \stackrel{\rho}\longrightarrow \e \cdot \x
	\xrightarrow{\tau\otimes 1}
		\e^!\cdot \x \xrightarrow{1\otimes \varphi}
	 	\e^!\cdot \e^! \stackrel{\mu}\longrightarrow \e^!,\]
	 	and similarly for $\varphi\frown\tau$. 
\end{theorem}

Moreover, we have 
the following computational result:

\begin{theorem}\label{diffCC}
The differential of the Koszul complex $K^*(\x)$ is such that if 
$ \varphi : \x[p] \longrightarrow \sgn_p$ is $S_p$-equivariant, then 
$df: \x[p+1]\longrightarrow \sgn_{p+1}$ is the $S_{p+1}$-equivariant
map so that for every $z\in \x[p]$,
\[ 
df(z) = \sum_{i=1}^{p+1} 
	(-1)^{i-1} \left( f(z\bs (I\smallsetminus i))-f(z\fs (I\smallsetminus i))\right).
\]
It follows that if $\x$ is a linearization $\kk\x_0$, the value of 
$df(z)$ for $f\in \SC^p(\x)$ and an element $z\in  \x_0(p+1)$ 
depends only on data that is degree-wise finite if $\x$ is of finite type.  
\end{theorem}

\begin{proof}
Let us write down $\tau\frown\varphi$ explicitly. Given some
$z\in\x[I]$, the definition of $\tau$ shows that
the composition $(\tau\otimes 1)\rho$ is obtained by
keeping in $\rho(z)$ only those summands where the left hand side
is supported in cardinality one. This gives us the sum
\[ \sum_{i\in I} s^ie_i\otimes z\bs (I\smallsetminus i) .\]
Applying $\varphi$ and then the product of $\e^!$
creates the sign
of the statement in the theorem. Similarly, the term $\varphi\frown\tau$
gives a sum
\[ \sum_{i\in I} z\fs (I\smallsetminus i) \otimes s^i e_i \]
where now putting $i$ in the correct position goes past $|I|-i+1$ 
generators $e_j$. Since the
degree of $\varphi$ is $|I|-1$, the previous sign
count and the Koszul sign rule give us the second term of the
sum with its sign.
\end{proof}

\subsection{Multiplicative matters.}
\label{sect:kunneth}

We now describe how to exploit the Koszul complex to deduce a K\"unneth theorem for the Cauchy product. 
\begin{prop}
For each such $p,q\in\mathbb N$ there is an isomorphism
\[\phi^p :  \hom_{S_p\times S_q}(\x\inter{p}\otimes \y\inter{q} ,\sgn_p\otimes \sgn_q) 
 	\longrightarrow \hom_{S_{p+q}}((\x\otimes\y)^p[p+q],\sgn_{p+q})
 		\]
where $(\x\otimes\y)^p\inter{p+q}$ is the space of summands
$\x[S]\otimes \y[T]$ with $S$ of cardinality $p$.
 \end{prop}
\begin{proof}
 This is readily described as follows. For each decomposition $(S,T)$ of $n$,
let $u = u_{S,T}$ be the unique bijection that assigns $S$ to $\inter{p}$
and $T$ to $\inter{p+1,p+q}$ in a monotone fashion, and given an element
$f\in\hom_{p,q}(\x\inter{p}\otimes \y\inter{q} ,\sgn_p\otimes \sgn_q)$, set
\[ \phi^p(f)(z\otimes w) = (-1)^{\sch(S,T)} f(z'\otimes w') \]
where $uz=z',uw=w'$ and $z\otimes w\in \x[S]\otimes \y(T)$.
We claim this is $S_{p+q}$-equivariant. Note that the sign of $u$ is $\sch(S,T)$. Indeed, if $\tau$
is a permutation of $n$ and $(S,T)$ is a decomposition
of $n$, we can write $\tau = \xi \rho$ where $\rho = \tau_1\times \tau_2$
is a shuffle of $(S,T)$ and $\xi$ is monotone over $S$ and over $T$.
It is clear that if $(S',T')$ is the image of $(S,T)$ under $\tau$
and if $u' = u'_{S',T'}$ then $u = u'\xi$. Moreover, note that
$u(\tau_1z\otimes \tau_2 w)$ is transported to $u(z\otimes w)$ by 
$u\rho^{-1}u^{-1}$, which belongs to $S_p\times S_q$, and we now 
compute 
\begin{align}
(-1)^\tau \phi^p(f)(\tau(z\otimes w))
						&= (-1)^{\tau+ u'} f(u'\tau(z\otimes w)) \\
						&= (-1)^{\tau+ u'} f(u (\tau_1 z \otimes \tau_2 w))\\
						&= (-1)^{\tau+ u'+\rho} f(u ( z \otimes  w)) \\
						&= (-1)^{\tau+ u'+\rho+u} \phi^p(f)(z \otimes  w) \\
						&= \phi^p(f)(z \otimes  w)
\end{align} 
where the signs cancel by virtue of the identities $\xi\rho = \tau$ and $u'\xi = u$.
\end{proof}

For each $p,q\in\mathbb N$ there are canonical maps 
\[ \label{eq:splithom}
 			\hom_{S_p}(\x\inter{p},\sgn_p)\otimes \hom_{S_q}(\y\inter{q},\sgn_q)
 					\longrightarrow 
			 \hom_{S_{p,q}}(\x\inter{p}\otimes \y\inter{q},\sgn_p\otimes \sgn_q) 
 	\]
that are all isomorphisms if $k$ is a field and $\x$ or $\y$ is finite in
each arity, and they collect along with the maps $\phi$ to
define a map
\[ \label{eq:kunnethCC} 
			-\times - :\SC^*(\x)\otimes \SC^*(\y)
				\longrightarrow 
				\SC^*(\x\cdot \y). \]
Explicitly, given maps $f_p : \x[p]\longrightarrow \sgn_p$ and
$g_q : \y(q) \longrightarrow \sgn_q$, we have for each
decomposition $(S,T)$ and $z\otimes w\in \x[S]\otimes \y(T)$ 
\[ (f_p \times g_q)(z\otimes w) 
	= (-1)^{\sch(S,T)} f_p(u_S(z))\otimes g_q(u_T(w)), \]
where $u = u_{S,T}$. We obtain now the main result
of this section, a K\"unneth formula that allows us
to compute the cohomology groups of a product in terms
of its factors.

\begin{theorem}[K\"unneth formula] \label{thm:Kunneth}
Suppose that $k$ is a field of characteristic zero and $\x$ or $\y$ is locally finite. The map $-\times - : \SC^*(\x)\otimes\SC^*(\y)\to\SC^*(\x\cdot\y)$ is
an isomorphism of complexes.  \qed
\end{theorem}

If $S$ is a subset of $[n]=\{1,\ldots,n\}$ 
with $m\leqslant n$ elements, and if 
$\sigma$ is a permutation of $S$, we 
regard $\sigma$ as a permutation of $[m]$ 
by means of the unique order preserving 
bijection $\lambda_S : S\longrightarrow 
[m]$. We say $(\sigma^1,\sigma^2)$ is a 
\emph{$(p,q)$-shuffle} of a finite set 
$I$ with $p+q$ elements whenever $\sigma^1$
is a permutation of a $p$-subset $S$ 
of $I$, $\sigma^2$ is a permutation of a 
$q$-subset $T$ of $I$, and $S\cap 
T=\varnothing$. Call $(S,T)$ the 
\emph{associated composition} of such a 
shuffle. If $(S,T)$ is a composition of 
$[n]$, we will write $\sch(S,T)$ for the 
\emph{Schubert statistic} of $(S,T)$, 
which counts the number of pairs $(s,t)
\in S\times T$ such that $s<t$ according
to the canonical ordering of $[n]$. Our 
result is the following

\begin{theorem}\label{cupformula} The cup product induced by the diagonal $\x\longrightarrow \x\cdot \x$ 
\[ - \smile - : \SC^p(\x)\otimes 
\SC^q(\x) \longrightarrow \SC^{p+q}(\x)\] is such that for equivariant maps $f : 
\x[p]\longrightarrow \sgn_p$ and $g:\x[q]\longrightarrow \sgn_q$, and $z\in \x(p+q)$,
\[ (f\smile g)(z) =  \sum_{(S,T)\vdash [p+q] } (-1)^{\sch(S,T)} f(\lambda_S(z\bs 
S)) g(\lambda_T(z\fs T))\] 
where the sum runs through decompositions of $[p+q]$ with $\# S=p$ and $\# T = 
q$. 
\end{theorem}

Before giving the proof, we begin with a few preliminary considerations. First, 
consider a $(p,q)$-shuffle $(\sigma^1,\sigma^2)$ of $[p+q]$, with associated 
composition $(S,T)$, and let $\sigma$ be the permutation of $[p+q]$ obtained by 
concatenating $\sigma^1$ and $\sigma^2$. 

\begin{lema}\label{schstat} For any $\sigma\in S_{p+q}$ and any 
$(p,q)$-composition $(S,T)$ of $[p+q]$, 
\begin{enumerate}
\setlength{\itemsep}{0pt}
  \setlength{\parskip}{0pt}
\item the sign of $\sigma$ is $(-1)^{\sigma^1+\sigma^2+\sch(S,T)}$, and
\item $(-1)^{\sch(S,T)} = (-1)^{\sch(T,S)+pq}$. 
\end{enumerate}
\end{lema}
\begin{proof}
Indeed, by counting inversions, it follows that the number of inversions in 
$\sigma$ is precisely $\operatorname{inv}\sigma^1+\operatorname{inv} 
\sigma^2+\sch(S,T)$, which proves the first assertion. The second claims follows 
from the first and the fact $\sigma^1\sigma^2$ and $\sigma^2\sigma^1$ differ by 
 $pq$ transpositions.  
\end{proof}

If $\alpha : \x[p] \longrightarrow \e^{\cdot p}$ is a cochain, we 
associate to it the equivariant map $f : \x[p] \longrightarrow \sgn_p$ such that 
$f(z) = \alpha(\nu_p)(z)$ where $\nu_p =
\sum_{\sigma\in S_p} (-1)^\sigma \sigma$ is the antisymmetrization element. Conversely, given such an equivariant 
map, we associate to it the cochain $\alpha : \x[p] \longrightarrow \e^{\cdot p}$ 
such that $\alpha(\sigma)(z) = \frac{(-1)^\sigma}{p!} f(z)$. We now proceed to 
the proof of Theorem~\ref{cupformula}.

\begin{proof} To calculate a representative of the class of $f\smile g$, we lift 
first lift the maps $f:\x[p]\longrightarrow \sgn_p $ and $g : \x[q] \longrightarrow \sgn_q$ 
to cochains $\alpha : \x \longrightarrow \e^{\cdot p},\beta: \x\longrightarrow \e^{\cdot 
q}$ that are supported in $\x[p]$ and $\x[q]$ respectively, and represent $f$ and 
$g$ according to the correspondence in the previous paragraph. We compute for any decomposition $(F^1,F^2)$ of a finite set $I$ 
and any $z\in\x[I]$ that
\[ (\alpha\smile \beta)(F^1,F^2)(z) = \alpha(F^1)(z\bs F^1)\beta(F^2)(z\fs 
F^2).\]
Now consider $z\in \x[p+q]$. If $\sigma$ is a permutation of $[p+q]$, write 
$(\sigma^1,\sigma^2)$ for the $(p,q)$-shuffle obtained by reading $\sigma(1) 
\cdots \sigma(p)$ as a permutation of $S_\sigma = \{\sigma (1),\ldots,\sigma 
(p)\}$ and by reading $\sigma(p+1)\cdots \sigma(p+q)$ as a permutation of 
$T_\sigma = \{\sigma(p+1),\ldots,\sigma (p+q)\}$. Then 
\begin{align*}
(f\smile g)(z) &=  \sum_{\sigma\in S_{p+q}} 
(-1)^\sigma(\alpha\smile\beta)(\sigma)(z)\\
			   &= \sum_{\sigma \in S_{p+q}}(-1)^\sigma \alpha(\sigma^1)(z\bs S_\sigma) 
\beta(\sigma^2)(z\fs T_\sigma)
			   \end{align*}
Fix a composition $(S,T)$ of $[p+q]$. In the sum above, the permutations 
$\sigma$ with $(S_\sigma,T_\sigma)=(S,T)$ are the $(p,q)$-shuffles with 
associated composition $(S,T)$. We may then replace the sum throughout $S_{p+q}$ 
with the sum throughout $(p,q)$-compositions $(S,T)$ of $[p+q]$ and in turn with 
the sum throughout shuffles $(\sigma^1,\sigma^2)$ of $(S,T)$. This reads
\[ (f\smile g)(z) =  \sum_{(S,T) \vdash [p+q]} \sum_{(\sigma^1,\sigma^2)}  
(-1)^{\sigma^1\sigma^2} \alpha(\sigma^1)(z\bs S)\beta(\sigma^2)(z\fs T). \] 
We now note that $\alpha(\sigma^1)(z\bs S) = 
\alpha(\lambda_S(\sigma^1))(\lambda_S(z\bs S))$, that the sign of 
$\lambda_S(\sigma^1)\in S_p$ is $(-1)^{\sigma^1}$, and that the same 
considerations apply to $\beta$, so we obtain that 
\[ (f\smile g)(z) = \frac{1}{p!q!}\sum_{(S,T) \vdash [p+q]} 
\sum_{(\sigma^1,\sigma^2)}  (-1)^{\sigma^1\sigma^2+\sigma^1+\sigma^2} 
f(\lambda_S(z\bs S))g(\lambda_T(z\fs T)) .\]
Using Lemma~\ref{schstat} finishes the proof: the sum 
$\sum_{(\sigma^1,\sigma^2)}  (-1)^{\sigma^1\sigma^2+\sigma^1+\sigma^2}$ 
consists of $p!q!$ instances of $(-1)^{\sch(S,T)}$.  \end{proof}

Suppose now that $\x$ is a cosymmetric $\e$-bicomodule. Then Theorem~\ref{diffCC} 
proves the differential in $\SC^*(\x)$ is trivial, while Lemma~\ref{schstat} 
along with Proposition~\ref{cupformula} prove that the cup product in 
$\SC^*(\x)$ is graded commutative. 
We have obtained the following result:

\begin{theorem} Suppose that $\x$ is a cosymmetric $\e$-bicomodule. Then 
$\SC^*(\x)$ is isomorphic to the cohomology algebra $H^*(\x)$. In particular, $H^*(\x)$ 
is graded commutative.  \qed
\end{theorem}

\subsection{Some computations.}\label{furthercomp} 

To illustrate the use of the Koszul complex we
compute the cohomology groups of some of the 
coalgebras considered by Aguiar and Mahajan and, in doing so, try to 
convince the reader of the usefulness of the complex we obtained.
To begin with, we include a new computation that 
is greatly simplified with the use of 
the small complex. We remark that,
as far as the author knows, the only
complete computation of such cohomology groups 
that was known before the
methods in this paper were introduced,
are $H^*(\spe{E})$ and the first
two cohomology groups of~$H^*(\spe{L})$,
along with a single cohomology class in $H^2(\L)$.
Throughout, we will use the following fact:

\begin{obs}
If $\kk$ is not of characteristic two, then
a functional $f\in \hom_{S_p}(\x[p],\sgn_p)$
vanishes on every $z\in \x[p]$ that is fixed by
an odd permutation. 
\end{obs}

Indeed, we see that if $\sigma z=z$ and if $\sigma$ is
odd then $-f(z) = f(\sigma z) = f(z)$, from where the
claim follows. 
\bigskip

\textit{The species of singletons and suspension.} Define the species $s$ of singletons so that for every finite set $I$, $\spe{S}[I]$ is 
trivial whenever $I$ is not a singleton, and is $\kk$-free with basis $I$ if $I$ 
is a singleton. The species $\spe{S}$ admits a unique $\e$-bicomodule structure. By 
induction, it is easy to check that, for each integer $q\geqslant 1$, the 
species $\spe{S}^{\tt q}$, which we write more simply by $\spe{S}^q$, is such that $\spe{S}^q[I]$ 
is $\kk$-free of dimension $q!$ if $I$ has $q$ elements with basis the linear 
orders on $I$, and the action of the symmetric group on $I$ is the regular 
representation, while $\spe{S}^q[I]$ is trivial in any other case. By convention, set 
$\spe{S}^0=\mathbb{1}$, the unit species. It follows that the sequence of species 
$\spe{S}=(\spe{S}^n)_{n\geqslant 0}$ consists of weakly projective species, and 
we can completely describe their cohomology groups. They are the analogues of 
spheres for species, its first property consisting of having cohomology 
concentrated in the right dimension:

\begin{prop} For each integer $n\geqslant 0$, the species $\spe{S}^n$ has 
\[ H^q(\spe{S}^n) = 
	\begin{cases}
		\kk & \text{if $q=n$ }, \\
		0 & \text{else}.
			\end{cases}
		\]
\end{prop}
\begin{proof}
Fix $n\geqslant 0$. By the remarks preceding the proposition, it follows that 
$\SC^q(\spe{S}^n)$ always vanishes except when $q=n$, where it equals 
$\hom_{S_n}(\kk S_n,\sgn_n)$, and this is one dimensional. Because each $\spe{S}^n$ is 
weakly projective, $\SC^*(\spe{S}^n)$ calculates $H^*(\spe{S}^n)$, and the claim follows.
\end{proof}

The above motivates us to check 
whether $\spe{S}\cdot -$ acts as a suspension for $H^*(-)$. Assume that $\x$ is 
weakly projective, so we may use $\SC^*(\x)$ to compute $H^*(\x)$. We claim 
that $\SC^*(\spe{S}\cdot\x)$ identifies with $\SC^*(\x)[-1]$. Indeed, for this it suffices 
to note, first, that $(\spe{S}\cdot \x)(n)$ is isomorphic, as an $\kk S_n$-module, to the 
induced representation $\kk\tt \x(n-1)$ from the inclusion $S_1\times S_{n-1} 
\hookrightarrow S_n$, and second, that the restriction of the sign 
representation of $S_n$ under this inclusion is the sign representation of 
the subgroup $S_{n-1}$, so that:
\begin{align*}
\hom_{S_n}((\spe{S}\cdot\x)[n],\sgn_n) &= \hom_{S_n}(\operatorname{Ind}_{S_1\times 
S_{n-1}}^{S_n}(\kk\tt \x[n-1]),\sgn_n))  \\
 	&\simeq \hom_{S_1\times S_{n-1}}(\kk\tt \x[n-1] , \operatorname{Res}_{S_1\times 
S_{n-1}}^{S_n} \sgn_n)\\
 	& \simeq \hom_{S_{n-1}}(\x[n-1],\sgn_{n-1}).
\end{align*}
A bit more of a calculation shows the differentials are the correct ones. By 
induction, of course, we obtain that $\spe{S}^j\cdot\x$ has the cohomology of $\x$, only 
moved $j$ places up.

\begin{prop}
Assume $\x$ is weakly projective. For each $j$,
the suspension $s^j \x=\spe{S}^j\cdot \x$ is also weakly projective,
and there is a natural suspension isomorphism
$ H^*(s^j\x) \longrightarrow  s^j H^*(\x)$
in cohomology groups. \qed
\end{prop}

\textit{The exponential species.}
Every structure on a set of cardinal larger than $1$ over the exponential 
species $\e$ is fixed by an odd permutation: if $I$ is a finite set with more 
than one element, there is a transposition $I\longrightarrow I$, and it fixes 
$e_I$. It follows that $\SC^q(\e)$ is zero for $q>1$, and it is immediate 
that $\SC^0(\e)$ and $\SC^1(\e)$ are one dimensional, while we already know 
$d=0$. Thus $H^q(\e)$ is zero for $q>1$ and is isomorphic to $\kk$ for $q\in 
\{0,1\}$. The cup product is then completely determined. This is in line with the computations
done in the thesis~\cite{tesisJavier} of J. Coppola:

\begin{prop}[J. Coppola]
The cohomology algebra of $\e$ is isomorphic to
the exterior algebra $k[t]/(t^2)$ in one generator. 
\qed
\end{prop}

\textit{The species of partitions.}  
The species of partitions $\spe{P}$ assigns to each finite set $I$ the collection 
$\spe{P}[I]$ of partitions $\x$ of $I$, that is, families $\{X_1,\ldots,X_t\}$ of 
disjoint non-empty subsets of $I$ whose union is $I$. There is a left 
$\e$-comodule structure on $\spe{P}$ defined as follows: if $X$ is a partition of $I$ 
and $S\subset I$, $X\ps S$ is the partition of $S$ obtained from the non-empty 
blocks of $\{ x\cap S :x\in X\}$. We already
noted there is an inclusion $\e\longrightarrow \spe{P}$. 
One can check that for \emph{any} bicomodule $\x$ we have a cocycle
$\kappa : \x\longrightarrow \e$ such that $\kappa(I)(z) = |I|e_I$,
which we call the \emph{cardinality cocycle.} In case $\x$ is
cosymmetric, one can check this is not a boundary.
 
\begin{prop} 
The cohomology group $H^0(\spe{P})$ is free of rank one, and $H^1(\spe{P})$ is free 
of rank one generated by the cardinality cocycle. In fact, the inclusion 
$\e\longrightarrow \spe{P}$ induces an isomorphism of
commutative algebras $\SC^*(\spe{P}) \longrightarrow 
\SC^*(\e)$.  
\end{prop}
\begin{proof}
A partition of a set with at least two elements is fixed by a transposition, and 
this implies that $\SC^j(\spe{P})=0$ for 
$j\geqslant 2$. On the other hand, $\SC^0(\spe{P})$ and $\SC^1(\spe{P})$ are both 
$\kk$-free of rank one, and we already know from Proposition~\ref{diffCC} that the 
differential of $\SC^*(\spe{P})$ is zero. This proves both claims. 
\end{proof}

\textit{The species of linear orders.}
The set species of linear orders $\spe{L}_0$ assigns to each finite
set $I$ set set $\spe{L}_0[I]$ of linear (also called total) orders
on $I$. It admits a cocommutative set coalgebra structure for which 
$\spe{L}_0[I] \longrightarrow \spe{L}_0[S]$ assigns a linear order on
$I$ to its restriction on $S\subseteq I$.
The resulting $\kk S_j$-module $\spe{L}[j]$ is the regular
representation of $S_j$ for each $j\geqslant 0$, and it
follows that the $\kk$-module $\hom_{S_j}(\spe{L}[j],\sgn_j)$ is $\kk$-free 
of rank one. Since $\L$ is cocommutative, by
virtue of Theorem~\ref{diffCC}, the computation ends here: 
the differential on this Koszul complex is identically zero.
Each $\SC^j(\L)$ is 
one dimensional generated by the map $f_j : \L[j]\longrightarrow \kk$ that assigns 
$\sigma\longmapsto (-1)^\sigma$. A calculation, which we omit, shows the following. It is useful to note
that the element $f_1$ corresponds to the cardinality
cocycle of $\spe{L}$ and that $f_2$ corresponds
to the Schubert cocycle defined by Aguiar and Mahajan~\cite{Aguiar}*{Section 9.7.1}.

\begin{prop} The algebra $\SC^*(\spe{L})$ is generated by the elements $f_1$ and 
$f_2$, so that if $f_p$ is the generator of $\SC^p(\spe{L})$, we have 
\begin{align*}
& f_{2p} \smile f_{2q} = \binom{p+q}p f_{2(p+q)}, && f_1\smile f_{2p} = 
f_{2p+1}, && f_1\smile f_{2p+1} = 0.
\end{align*}
These relations	 exhibit $H^*(\spe{L})$ as a tensor product of a divided power 
algebra and an exterior algebra. 	
 \qed
\end{prop}

This in turn fixes a small error in the literature \cite{Aguiar}*{Proposition 9.28}, which claims there is another $2$-cocycle
in $H^2(\L)$: unfortunately, the `descent cocycle' defined
there does not satisfy the cocycle equation. 
We will address the multiplicative
structure of $H^*(\spe{L})$ below.

\bigskip

\textit{The species of compositions.} 
The species of compositions $\spe{C}$ is the non-abelian analogue of the species 
of partitions $\spe{P}$. Let us recall its construction: the species of compositions 
$\spe{C}$ assigns to each finite set $I$ the set $\spe{C}[I]$ of compositions of 
$I$, that is, ordered tuples $(F_1,\ldots,F_t)$ of disjoint non-empty subsets of 
$I$ whose union is $I$. This has a standard left $\e$-comodule structure such 
that if $F=(F_1,\ldots,F_t)$ is a composition of $I$ and $S\subseteq I$, $F\ps 
S$ is the composition of $S$ obtained from the tuple $(F_1\cap S,\ldots,F_t\cap 
S)$ by deleting empty blocks. We view $\spe{C}$ as an $\e$-bicomodule with its 
cosymmetric structure. 

\begin{prop} The morphism $\spe{L}\longrightarrow \spe{C}$ induces an isomorphism
$H^*(\spe{C})\longrightarrow H^*(\spe{L})$ and, in fact, an isomorphism of commutative algebras
$\SC^*(\spe{C})\longrightarrow \SC^*(\spe{L})$. 
\end{prop}
\begin{proof} It suffices that we prove the second claim, and, since $\SC^*(-)$ 
is a functor, that for a fixed integer $q$, the map 
$\SC^q(\spe{C})\longrightarrow \SC^q(\spe{L})$ is an isomorphism of modules. This 
follows from the fact a decomposition $F$ of a set $I$ is fixed by 
a transposition as soon as it has a block with at least two elements, and 
therefore an element of $\SC^q(\spe{C})$ vanishes on every composition of $[q]$, 
except possibly on those into singletons. Thus the surjective map 
$\SC^*(\spe{C})\longrightarrow \SC^*(\spe{L})$ is injective and it is thus an isomorphism of commutative
algebras. 
\end{proof}

\textit{The species of graphs.} 
The species 
$\mathsf{Gr}$ of graphs that assigns to 
each finite set $I$ the collection of graphs
with vertices labeled by $I$ admits a cosymmetric
$\e$-bicomodule structure obtained by restriction
of graphs. We 
have the following result concerning the cohomology groups of $\mathsf{Gr}$:

\begin{prop}
If $\Bbbk$ is of characteristic zero then, for each non-negative integer $p\geqslant 
0$, $\dim_k H^p(\mathsf{Gr})$ equals the number of isomorphism classes of 
graphs on $p$ vertices with no odd automorphisms, namely, 
\[ 1,\quad 1,\quad  0,\quad  0,\quad  1,\quad  6,\quad  28,\quad  252,\quad  
4726,\quad  150324,\quad  \ldots \] 
\end{prop} 
This sequence is \cite{OEIS}*{\href{http://oeis.org/A281003}{A281003}}. 
\begin{proof}
Since the structure on $\mathsf{Gr}$ is cosymmetric, the differential of 
$\SC^*(\mathsf{Gr})$ vanishes, and since we are mapping to the
sign representation, this tells us 
$\SC^q(\mathsf{Gr})$ has dimension as in the statement of the proposition. The 
tabulation of the isomorphism classes of graphs in low cardinalities can be done 
with the aid of a computer ---we refer to Brendan McKay's calculation 
\cite{McKay} for the final result--- and then filter out those graphs with odd 
automorphisms. 
\end{proof}

We can exhibit cocycles whose cohomology classes generate  $H^1(\mathsf{Gr})$ 
and $H^4(\mathsf{Gr})$: in degree one, we have the cardinality cocyle 
$\kappa$, and in degree four, the normalized cochain $p^4: \mathsf{Gr} 
\longrightarrow \e^{\tt 4}$ such that for a decomposition $F\vdash_4 I$, and a 
graph $g$ with vertices on $I$, $p^4(F_1,F_2,F_3,F_4)(g)$ is the number of 
inclusions $\zeta : p_4\longrightarrow g$, where $p_4$ is the graph
\begin{center}
\begin{tikzpicture}[line cap=round,line join=round,>=triangle 
45,x=1.0cm,y=1.0cm]
\clip(0.5,1.5) rectangle (4.25,2.5);
\draw (1.,2.)-- (2.,2.);
\draw (2.,2.)-- (3.,2.);
\draw (3.,2.)-- (4.,2.);
\begin{scriptsize}
\draw [fill=black] (1.,2.) circle (1.5pt) node[above]{$1$};
\draw [fill=black] (2.,2.) circle (1.5pt) node[above]{$2$};
\draw [fill=black] (3.,2.) circle (1.5pt) node[above]{$3$};
\draw [fill=black] (4.,2.) circle (1.5pt) node[above]{$4$};
\end{scriptsize}
\end{tikzpicture}
\end{center}
and $\zeta(i) \in F_i$ for $i\in [4]$. One can check this cochain is in fact a 
cocycle, and it is normalized by construction. 
 Even more can be 
said: our formula for the cup product and induction shows that for each 
$n\geqslant 1$, the product $f^n$ is nonzero on the graph that is the disjoint 
union of $n$ paths $p_4$, so that $H^{4n}$ is always non-vanishing for 
$n\geqslant 1$. Hence the cohomology algebra $H^*(\mathsf{Gr})$ contains both 
an exterior algebra in degree $1$ and a polynomial algebra in degree four.

\section{Koszul duality for (twisted) coalgebras}

Let us fix a (twisted) coalgebra $\cc$.
The work done above carries over, mutatis mutandis,
to the case $\cc$ is Koszul. When this is the case, for
every $\cc$-bicomodule we have
available the Koszul complex 
\[ K^*(\x,\cc) = 
	\hom_{\Sp_\kk}(\x,\cc^\antishriek) \] 
	to compute its cohomology groups $H^*(\x,\cc)$. The 
	structure of certain (co)algebras arising
	from combinatorial objects (such as polytopes)
	can be nicely 
	understood through combinatorial methods, see for
	example~\cite{AguiarArdila}.
	We also remark that in the book~\cite{VladimirMurray}*{Chapter 4}, V. Dotsenko and M. Bremner
	explain how apply methods of Gr\"obner bases
	to twisted algebras, which can then be
	effectively used to obtain results on
	the Koszulness of these. In particular,
	their formalism implies immediately that
	every free cocommutative (twisted) coalgebra 
 	is Koszul, by repeating, mutatis mutandis,
 	the proof done in the classical case.
 	
	\begin{remark}
	It very often happens that $\cc$ 
	is Koszul for trivial reasons: if $\cc$ is a cocommutative connected (twisted) bialgebra, then by the Milnor--Moore theorem (see \cite{AguiarHopf}*{Theorem 118}
	and \cite{Stover}*{Theorem 11.3}) the underlying coalgebra of $\cc$ is cocommutative cofree over 
	the collection of its primitives. 
	\end{remark}
	
	We can apply this remark to the species
	$\spe{L}$ of linear orders to deduce
	the following result which should
	aid us in computing $H^*(-,\spe{L})$ in the category of $\spe{L}$-bicomodules.
	
	\begin{cor}[cf. \cite{Patras2004}*{Proposition 10 and Proposition 17}] The coalgebra
	$\spe{L}$ is cofree cocommutative conilpotent
	over the species underlying the
	free Lie algebra functor and,
	in particular, it is Koszul with Koszul dual  algebra 
	$ \spe{L}^{\emph{\antishriek}}= S(s^{-1}\mathsf{Lie})$, 
	the free commutative
	algebra over the desuspension of $\;\mathsf{Lie}$. 
	\end{cor}
	
	\begin{proof}
	The fact that the primitives of
	$\spe{L}$ is equal to $\mathsf{Lie}$ 
	contained in \cite{AguiarHopf}*{Corollary 121} and the article cited in the statement of the corollary, and by
	the Milnor--Moore theorem in the
	remark above we have have isomorphism
	$ S^c(\mathsf{Lie}) \longrightarrow \spe{L}$
	 of cocommutative conilpotent coalgebras. Since  
	cocommutative coalgebras are Koszul, the result follows.	
	\end{proof}
	
 With this at hand, we can prove the
 following result.
 
 \begin{theorem}
 The  Koszul dual algebra
 of the  coalgebra $\spe{L}$ is
 the free commutative algebra
 $S(s^{-1}\mathsf{Lie})$ over the desuspension of $\;\mathsf{Lie}$. In particular, for each $p,q\in\NN$
 we have that
 \[ H^{p-q}(\Omega^*(\spe{L})[p])=S(s^{-1}\mathsf{Lie})[p]^{p-q}
 \]
 has basis in correspondence with
 the permutations of $p$ with
 $p-q$ disjoint cycles and, hence,\[
 \dim_k S(s^{-1}\mathsf{Lie})[p]^{p-q} =  \stirlingi{p}{p-q}\]
 an unsigned Stirling number of the first kind.\end{theorem}
 
 \begin{proof}
  The first equality follows by
  Koszul duality, since the algebra
  $S(s^{-1}\mathsf{Lie})$ is Koszul 
  dual to $\spe{L}$ and, as such
  equal to the homology of the cobar
  construction on $\spe{L}$.
  For the second equality, all that we need to do is observe
 that for a finite set $I$ of size
 $p$, 
 an elementary tensor
  \[z_1\odot\cdots\odot z_{p-q}\in S(s^{-1}\mathsf{Lie})[p]^{p-q} \]
    of $S(s^{-1}\mathsf{Lie})$ is
    in homological degree $p-q$
    and corresponds to the datum of
    an \emph{unordered} partition 
    of $I$ into subsets
    $(F_1,\ldots,F_{p-q})$
    with $z_i \in \mathsf{Lie}[F_i]$. 
    On the other hand, for a finite
    set $[n]$ we have a basis of
    $\mathsf{Lie}[n]$ indexed by
    permutations of $n$ that fix 
    $1$. In this way, such an elementary
    tensor 
        is indeed in bijection with a
    permutation of $I$ with $p-q$
    disjoint cycles, which is what
    we wanted. 
 \end{proof}
 
 As pointed out to the author, the previous theorem
 computes the underlying $\kk$-spaces of the Koszul
dual algebra $H(\Omega(\L)) = \L^!$, but does not compute
the symmetric group actions, which are necessary to carry out
the computation of the Koszul complex $\hom(\x,\L^!)$. It
would be interesting to obtain this description.

\section{A word on deformations}

Fix a coalgebra $\x$ in $\Sp_k$, with comultiplication $\Delta : \x\longrightarrow 
\x\cdot \x$ and counit $\varepsilon : \x\longrightarrow 1$, and set $K = 
\kk\llbracket t\rrbracket$, the algebra of formal power series in $\kk$. We write 
$\x\llbracket t\rrbracket $ for the coalgebra in $\Sp_K$ obtained by extending scalars pointwise in 
$\x$. By a \emph{weak deformation of $\x$} we mean a $K$-linear
coalgebra structure 
$\Delta_t : \x\llbracket t\rrbracket\longrightarrow (\x \cdot \x)\llbracket t\rrbracket$ determined on generators by 
\[ \Delta_t(S,T)(z) = \sum_{i\geqslant 0} \Delta_i(S,T)(z)
t^i \cdot   z\bs S \otimes z\fs T= f_t(S,T)\cdot   z\bs S \otimes z\fs T \]
where $\Delta_i : \x\longrightarrow \e^2$ and $\Delta_0=1$. 

\begin{remark}
These deformation are  
\emph{weak} in the sense we are only modifying the coefficients of our 
comultiplication by means of a power series coefficient and not modifying the 
higher order terms of the comultiplication: the terms $z\bs S\otimes z\fs T$ appearing in the formula above are the original ones
appearing in the comultiplication of $\x$. 
\end{remark}

As
mentioned in the introduction, generic deformations require us to consider
maps $\Delta_i : \x\longrightarrow \x^{\cdot 2}$ that are cochains in ${C}^2(\x,\x)$, but we are only allowing for cochains in ${C}^2(\x,\e)$. 
The condition that $\Delta_t$ is associative is equivalent to the collection of 
equalities
\begin{align*}
\tag{$\delta_n$} \sum_{i+j=n} \left(\Delta_i(RS,T)(z)\Delta_j(R,S)(z\bs T^c) - 
\Delta_i(R,ST)(z)\Delta_j(S,T)(z\fs R^c)\right)= 0
\end{align*}
for $n\in\mathbb{N}_0$ and $(R,S,T)$ an arbitrary decomposition of a finite set. 
The condition that $\Delta_t$ is counital es equivalent to each $\Delta_i$ being 
normalized for $i\geqslant 1$. In particular $(\delta_1)$ says that $\Delta_1$ 
belongs to $Z^2(\x,\e)$, as expected. Thus, every weak deformation of $\x$ 
gives a corresponding $2$-cocycle
representing an element in $H^2(\x,\e)$ . 

\smallskip

 Given cochains $\alpha,\beta\in C^2(\x,\e)$, we 
write $\alpha \star_{1,0} \beta$ and  $\alpha \star_{0,1} \beta$ 
for the $3$-cochains such that
\begin{align*}
(\alpha \star_{1,0} \beta)(R,S,T)(z) &= 
\alpha(R,ST)(z)\beta(S,T)(z \fs R^c),\\
(\alpha \star_{0,1} \beta)(R,S,T)(z) &= 
\alpha(RS,T)(z)\beta(R,S)(z\bs T^c). 
\end{align*}

In turn, we define a Gerstenhaber bracket using the
following formulas: 
\begin{align}\label{mixedops}
& \alpha\star\beta = \alpha  \star_{1,0} \beta - 
\alpha {\star}_{0,1} \beta, && \{ \alpha,\beta \} = \alpha\star\beta - 
\beta \star \alpha.
\end{align}
With this at hand, we can write equation $(\delta_n)$ in
the succinct form
\[ \sum_{i+j=n} \Delta_i \star \Delta_j  = 0.\] We say a $2$-cocycle $\Delta_1$ is integrable if it arises in this way from a 
deformation of $\x$. We note that from $(\delta_2)$ that if
$\Delta_1$ is integrable, then its first obstruction
$\sigma_1 =  \Delta_1 \star \Delta_1$
must me a coboundary in $C^3(\x,\e)$. A lengthy calculation shows that
the first obstruction $\sigma_1$ is a $3$-cocycle. 
One can use every $2$-cocycle to deform a coalgebra in the strong sense in characteristic zero which is, as we mentioned, the sense in 
which Aguiar and Mahajan intend to deform coalgebras in 
\cite{Aguiar}.

\begin{theorem}\label{thm:integration}
Suppose $\kk$ is of characteristic zero and let $\Delta_1$ be a 
normalized $2$-cocycle $\x\longrightarrow \e^2$. Then there exists 
a weak deformation $\x_t$ of $\x$ corresponding to $\Delta_1$.
\end{theorem}

\begin{proof} We set, for each $i\in \mathbb N_0$ and each decomposition $(S,T)$ 
of a finite set, 
 \[ \Delta_i(S,T)(z) = \frac{1}{i!}\Delta_1(S,T)(z)^i\] 
and observe that equation $(\delta_n)$ can be obtained considering the
equations (which hold because $\Delta_1$ is a $2$-cocycle)
\[ \Delta_1(RS,T)(c) + \Delta_1(R,S)(z\bs T^c) =  
   \Delta_1(R,ST)(c) + \Delta_1(S,T)(z\fs R^c) \]
and raising the left and right terms to the $n$th power. 
\end{proof}

Remark that, if we write $\Delta_t(S,T)(c) = \exp( \Delta_1(S,T)(c) t)$, then 
the analogy with the deformations considered in \cite{Aguiar}*{Section 9.6} is made evident by the 
change of variables $q = e^t$; in there the authors consider the simpler looking formula
\[ \Delta_q(S,T)(z) = q^{\Delta_1(S,T)}  z\bs S \otimes z\fs T.\] 
Moreover, what we did above shows one can 
integrate a cocycle up to degree $p-1$ when $\kk$ is of characteristic $p$, and  in such case the $p$th obstruction to the integrability of $\Delta_1$ is
\[\sigma_p = -\frac{1}{p}\sum_{i=1}^{p-1} \binom{p}{i} 
\{ \Delta^i,\Delta^{p-i}\}. \]
just as in~\cite{Gerstenhaber1964}*{p. 70}.


\bibliographystyle{alpha} 
\bibliography{twisted.bib}

\Addresses

\end{document}